\documentclass{amsart}
\usepackage{latexsym}
\usepackage{amsmath,amsfonts,epsfig, graphics, graphicx, amsthm, amssymb}
\usepackage{dsfont}
\usepackage[shortlabels]{enumitem}
\input{xy}
\xyoption{all}

\newtheorem{theorem}{Theorem}[section]
\newtheorem{lemma}[theorem]{Lemma}
\newtheorem{corollary}[theorem]{Corollary}

\newtheorem*{GRH}{Generalized Riemann Hypothesis}

\pdfoutput=1

\begin{document}
\title{Primes in arithmetic progressions and Siegel zeroes}
\author[T. Wright]{Thomas Wright}
\address{Wofford College\\429 N. Church St.\\Spartanburg, SC 29302\\USA}

\begin{abstract}
Let $\chi$ be a Dirichlet character mod $D$ with $L(s,\chi)$ its associated $L$-function, and let $\psi(x,q,a)$ be, as usual, Chebyshev's prime-counting function for the primes of the arithmetic progression $a$ (mod $q$) with $(a,q)=1$.  Let $\chi$ be a primitive character modulo $D$, and let $\nu>0$ be small.  We prove that if $L(s,\chi)$ has a Siegel zero at $s=\beta=1-\frac{1}{\eta\log D}$ with $\eta>\eta_0(\nu)$ for some large $\eta_0(\nu)$, there exists a range of $x$ for which the asymptotic
$$\psi(x,q,a)=\frac{\psi(x)}{\phi(q)}\left[1-\chi\left(\frac{aD}{(q,D)}\right)+O(\varepsilon_{\eta_0})\right]$$
holds for $q<x^{\frac{30}{59}-\nu}$.  We also show slightly better bounds for $q$ if we take an average over a range of $q$, finding an Elliott-Halberstam-type result for $q\sim Q$ on the range $Q<x^{\frac{16}{31}-\nu}$.
This improves on a 2003 result of Friedlander and Iwaniec that requires $q<x^{\frac{233}{462}}$ and builds on recent work of Sachpazis.
\end{abstract}
\maketitle

\section{Introduction}

In 1837, Peter Gustav Lejeune Dirichlet proved the prime number theorem in arithmetic progressions by introducing (in somewhat different notation) his eponymous $L$-function:
$$L(s,\chi)=\sum_{n=1}^\infty \frac{\chi(n)}{n^s}.$$
Here, $\chi$ is a Dirichlet character modulo an integer $q>2$.  We will assume that $\chi$ is non-principal, and hence the above sum is convergent for $Re(s)>0$.

Dirichlet's studies raised the question of when this function $L(s,\chi)$ equals zero.  In particular, the zero-free region around $s=1$ led to Dirichlet's theorem on the asymptotic density of prime numbers in arithmetic progressions, while larger zero-free regions would allow for better error terms for this theorem.  Indeed, one of the most famous conjectures in mathematics is the belief that all of these zeroes are, in fact, on the half-line:

\begin{GRH}
For a Dirichlet character $\chi$, let $L(s,\chi)=0$ for $s=\sigma+it$ with $\sigma>0$.  Then $\sigma=\frac 12$.
\end{GRH}

Of course, we are nowhere close to proving this.  In the case where the zero is real, the best effective bound comes from Landau's 1918 paper \cite{La}:
\begin{theorem}[Landau, 1918]
There exists an effectively computable positive constant $C$ such that for any $q$ and any character $\chi$ mod $q$, if $L(s,\chi)=0$ and $s$ is real, then
$$s<1-\frac{C}{q^\frac 12 \log^2 q}.$$
\end{theorem}
The best ineffective bound was proven by Siegel in 1935 \cite{Si}, who was able to improve the denominator in Landau's result to any $q^\varepsilon$ but at the cost of no longer being able to compute the associated constant:
\begin{theorem}[Siegel, 1935]
For any $\varepsilon>0$ there exists a positive constant $C(\varepsilon)$ such that if $L(s,\chi)=0$ and $s$ is real then
$$s<1-C(\varepsilon)q^{-\varepsilon}.$$
\end{theorem}
However, most zeroes are far closer to the half-line than these bounds indicate.  In fact, it is known (see \cite{Gr}, \cite{La}, \cite{Ti}) that for any $q$, every zero of $L(s,\chi)$ except at most one will obey a much smaller bound:
\begin{theorem}
There is an effectively computable positive constant $C$ such that $$\prod_{\chi\mbox{ }mod\mbox{ }q}L(s,\chi)=0$$
has at most one solution on the region $$\sigma\geq 1-\frac{C}{\log q(2+|t|)}.$$
If such a zero exists, $s$ must be real, and the character for which $L(s,\chi)=0$ must be a non-principal real character.
\end{theorem}


A zero of this type, if it is to exist, is called a \textit{Siegel zero} or an \textit{exceptional zero}, and the associated character is called an exceptional character. We note that the definition given here (or, indeed, in the literature in general) for a Siegel zero is not particularly rigorous, since this definition depends on the choice of the constant $C$.

\section{Siegel Zeroes}

While the existence of Siegel zeroes would unfortunately disprove the Riemann hypothesis, these zeroes would nevertheless lead to some surprisingly nice properties among the primes.  Most notably, the existence of Siegel zeroes would allow us to prove (among other things) the twin prime conjecture \cite{HB83}, small gaps between general $m$-tuples of primes \cite{WrS}, the existence of large intervals where the Goldbach conjecture is true \cite{MaMe}, a hybrid Chowla and Hardy-Littlewood conjecture \cite{TT}, and results about primes in arithmetic progressions that would allow the modulus $q$ to be greater than $\sqrt x$ \cite{FI03}.  It is this last result that is of interest in the present paper.

In the definitions below, we will assume that $(a,q)=1$.  We recall that Chebyshev's functions are given by
\begin{gather*}
\psi(x)=\sum_{n\leq x}\Lambda(n),\\
\psi(x,q,a)=\sum_{\substack{n\leq x\\n\equiv a\pmod q}}\Lambda(n),
\end{gather*}
where $\Lambda$ is the von Mangoldt function given by
$$\Lambda(n)=\begin{cases} \log p & if\mbox{ }n=p^k\mbox{ }for\mbox{ }prime\mbox{ }p,\\ 0 & otherwise.\end{cases}$$
In 2003, Friedlander and Iwaniec \cite{FI03} proved the following:

\begin{theorem}[Friedlander-Iwaniec, 2003] \label{FIFI}
Let $\chi$ be a real character mod $D$.  Let $x>D^r$ with $r=554,401$, let $q=x^\theta$ with $\theta<\frac{233}{462}$, and let $(a,q)=1$.  Then
$$\psi(x,q,a)=\frac{\psi(x)}{\phi(q)}\left(1-\chi\left(\frac{aD}{(q,D)}\right)+O\left(L(1,\chi)(\log x)^{r^r}\right)\right).$$
\end{theorem}

Notably, this allows for moduli $q$ that are larger than $x^{\frac 12}$.  In fact, their work actually proves this theorem for the slightly larger region of $$\theta<\frac{58\left(1-\frac 1r\right)}{115}$$
for some very large $r$.  The authors rounded off to an exponent of $\frac{233}{462}$ for the purpose of writing their result more simply.

In this paper, we increase further the allowable size for $q$.  A previous unpublished version of this paper also reduced the error term to 
$$O\left(L(1,\chi)(\log x)^{7}\right).$$
More recently, Sachpazis \cite{Sa} reduced the requirement of the Siegel zero further, albeit with the same $q$ as in the Friedlander-Iwaniec paper.  We state the results of that paper here.  

\begin{theorem}[Sachpazis] \label{SachThm}
Let $x\geq 2,\nu\in(0,1/100)$, and consider positive integers $a,q$ and $D$ such that $(a,q)=1,q\leq x^{58/115-\nu},$ and $x=D^V$ for some $V\geq 200/\nu$. Let also $\chi$ be a quadratic primitive character $\mod{D}$ and assume that $\eta_0=\eta_0(\nu)>0$ is a sufficiently large real number in terms of $\nu$. If for some $\eta\geq \eta_0$, there exists a real number $\beta=1-1/(\eta\log D)$ such that $L(\beta,\chi)=0$, then
\begin{align}\label{thmas}
\psi(x;q,a)=\frac{\psi(x)}{\phi(q)}\bigg\{1-\mathds{1}_{D\mid q}\chi(a)+O_{\nu}\bigg(\frac{V^{16}}{\eta}+\exp\Big(-C_{\nu}\sqrt{V\log\eta}\Big)\bigg)\bigg\},
\end{align}
where $C_{\nu}$ is a positive constant that depends on $\nu$.
\end{theorem}
Noting that $V=\log x/\log D$, we can write the error term as
$$O_\nu\left(\frac{(1-\beta)(\log x)^{16}}{(\log D)^{15}}+\exp\Big(-C_{\nu}\sqrt{V\log\eta}\Big)\right).$$
This leads to the following corollary:

\begin{corollary}\label{SachCor}
Under the same considerations and notation as in Theorem \ref{SachThm}, we have the following:
\begin{enumerate}[(a)]
    \item There exists a positive constant $C'_{\nu}$ depending on $\nu$ such that
\begin{align*}
\psi(x;q,a)=\frac{\psi(x)}{\phi(q)}\bigg\{1-\mathds{1}_{D\mid q}\chi(a)+O_{\nu}\bigg(\!\exp\Big(-C_{\nu}'\sqrt{V\log\eta}\Big)\bigg)\bigg\},
\end{align*}
     for every $x\leq D^{C_{\nu}^{-2}\log\eta}$.
     \item For every $\delta\in(0,1)$, we have
\begin{eqnarray*} 
\psi(x;q,a)=\frac{\psi(x)}{\phi(q)}\bigg\{1-\mathds{1}_{D\mid q}\chi(a)+O_{\nu}\bigg(\frac{1}{\eta^{1-\delta}}\bigg)\bigg\}
\end{eqnarray*}
      for all $x \in (D^{{C_{\nu}}^{-2}\log{\eta}},D^{\eta^{\delta/16}}]$.
\end{enumerate}
\end{corollary}

Since Sachpazis' paper has established the blueprint for how to reduce the requirements of the Siegel zero to 
$1-\beta\ll \log x$, we alter our previous unpublished paper to follow their blueprint in improving the bounds for $q$.


\section{Main Theorems}
In this paper, we increase the allowable range of $q$ in the following result.
\begin{theorem}\label{MainTheoremi}
Let $x$ be large, $\nu\in(0,1/100)$, $q\leq x^{30/59-\nu},$ and $x=D^V$ for some $V\geq 200/\nu$ as before. Let $\chi$ again be a primitive character mod $D$ such that $L(s,\chi)$ has an exceptional zero at $\beta=1-1/(\eta\log D)$.  Assume again that   $\eta_0=\eta_0(\nu)>0$ is a sufficiently large real number in terms of $\nu$. Then
\begin{align}
\psi(x;q,a)=\frac{\psi(x)}{\phi(q)}\bigg\{1-\mathds{1}_{D\mid q}\chi(a)+O_{\nu}\bigg(\frac{V^{16}}{\eta}+\exp\Big(-C_{\nu}\sqrt{V\log\eta}\Big)\bigg)\bigg\},
\end{align}
where $C_{\nu}$ is a positive constant that depends on $\nu$.
\end{theorem}
We can increase the allowable range of $q$ even further if we are willing to settle for an Elliott-Halberstam-type result.  As is standard, we let $n\sim x$ denote $x\leq n\leq 2x$.
\begin{theorem}\label{MainTheoremii}
Let $\chi$, $D$, $\nu$, and $x$ be as in Theorem \ref{MainTheoremi}, and let $Q<x^{\frac{16}{31}-\nu}$.  Then
$$\sum_{q\sim Q}\max_{(a,q)=1}\left|\psi(x,q,a)-\left(1-\chi\left(\frac{aD}{(q,D)}\right)\right)\frac{\psi(x)}{\phi(q)}\right|\ll \frac{xV^{16}}{\eta}+\exp\Big(-C_{\nu}\sqrt{V\log\eta}\Big)\bigg).$$
\end{theorem}
Notably, since
$$\sum_{\substack{q\sim Q\\ D|q}}\frac{\psi(x)}{\phi(q)}\ll \frac x{\phi(D)},$$
we can rewrite Theorem \ref{MainTheoremii} as conditional progress towards the Elliott-Halberstam conjecture.
\begin{corollary} Let $\chi$, $D$, $x$, and $Q$ be as in the previous theorem.  Then
$$\sum_{q\sim Q}\max_{(a,q)=1}\left|\psi(x,q,a)-\frac{\psi(x)}{\phi(q)}\right|\ll \frac{xV^{16}}{\eta}+\exp\Big(-C_{\nu}\sqrt{V\log\eta}\Big)\bigg).$$
\end{corollary}
We note that the analogous problem with a fixed congruence class $a$ has been dealt with previously, as the current author \cite{WrLM} found that $$\sum_{\substack{q\sim Q \\ (a,q)=1}}\left|\psi(x,q,a)-\frac{\psi(x)}{\phi(q)}\right|\ll xL(1,\chi)\log^{5} D$$
for $q<x^{\frac 23-\varepsilon}$.  However, the techniques involved in that paper do not apply here.


\section{Ideas for the Paper: Notation}

Let $\chi$ be an exceptional character of conductor $D$, and let $\ast$ denote the Dirichlet convolution.  Moreover, let $\mu$ denote the M\"{o}bius function, and recall that $$\Lambda(n)=(\mu\ast\log)(n),$$ where $\Lambda$ is the von Mangoldt function defined earlier.

Traditionally, questions about primes have tended to focus on the von Mangoldt function.  In \cite{FI03}, the authors' idea was that one can rewrite $\Lambda$ with $$\Lambda=\mu\ast\log\ast \chi\ast \chi\mu,$$ since $(\chi\ast \chi\mu)(n)$ is 1 if $n=1$ and zero otherwise.  Regrouping these terms, one has that
\begin{gather}\label{nulambda}
\Lambda=(\log \ast \chi)\ast(\mu\ast \chi\mu).
\end{gather}
The $\log\ast\chi$ term can be evaluated using standard $L$-function contour integration techniques and Weil's bound for Kloosterman sums.  Meanwhile, for the $\mu\ast \chi\mu$ term, one can see that $$|(\mu\ast \chi\mu)(n)|\leq (1\ast\chi)(n),$$and this, too, is easier to evaluate than the von Mangoldt function.

More specifically, define
\begin{gather*}
\lambda(n)=(\chi\ast 1)(n),\\
\lambda'(n)=(\chi\ast \log )(n),\end{gather*} and $$\nu'(n)=(\mu\ast (\mu \chi))(n).$$
Importantly, under the assumption of a Siegel zero, sums over $\lambda(n)$ are small.  In particular, if $x>D^2$, then
\begin{gather}\label{lambdaid1}
\sum_{d\leq x}\lambda(d)=xL(1,\chi)+O\left(\sqrt{Dx}\right),
\end{gather}
and
\begin{gather}\label{lambdaid2}
\sum_{D^2<d\leq x}\frac{\lambda(d)}{d}\ll L(1,\chi)\log x,
\end{gather}
which are Lemma 5.1 and equation (5.9) of \cite{FI03}, respectively.  For results on Siegel zeroes, these identities are a key point of leverage, as the assumption that $L(1,\chi)$ is small allows one to extract savings from these two bounds.

The relationship between $\lambda'$ and $\Lambda$ can be given by
\begin{gather}
\lambda'=\chi\ast \log =\chi\ast 1\ast \mu \ast \log =\lambda\ast \Lambda,
\end{gather}
and
\begin{gather}
\Lambda=\mu\ast \log =\chi\ast \chi\mu \ast \mu\ast \log =\nu\ast \lambda'.
\end{gather}
\section{Ideas for the Paper: the Friedlander-Iwaniec Framework}
This last identity can be used to re-express Chebyshev's function:
$$\psi(x,q,a)=\mathop{\sum\sum}\limits_{\substack{d,m\\dm\leq x\\ dm\equiv a\pmod q}}\nu(d)\lambda'(m).$$
Friedlander and Iwaniec split this double sum into two parts: the part where $d$ is small, and the rest.  In the former case, $\lambda'(m)$ can be evaluated directly with $\nu(d)$ having little impact, and it is this sum over $\lambda'(m)$ that gives the main term in their theorem.

In the case where $d$ is not small, the authors use the fact that we mentioned in our discussion of (\ref{nulambda}), namely that
\begin{gather}\label{lambdanubound}
|\nu'(d)|\leq \lambda(d).
\end{gather}
This allows them to write
$$|(\nu'\ast \lambda')(n)|\leq (\lambda\ast \lambda')(n)\leq \log(n)(\lambda\ast 1\ast 1)(n).$$
Unfortunately, there are few results that can help with an expression such as the one on the right, since this is a quaternary divisor function $\chi\ast 1\ast 1\ast 1$, and the only divisor functions where $q$ can be taken larger than $\sqrt x$ are binary ones like $1\ast 1$ or ternary ones like $1\ast 1\ast 1$.  To combat this, the authors use an inequality of Landreau \cite{Land} that essentially simplifies the expression to $\lambda\ast 1$.  This simplified expression is much more amenable to ternary sum methods, but this technique comes at the cost of a significantly worse bound and thus only helps under the assumption of a strong Siegel zero.  However, recent work of Sachpazis limited the sum to rough numbers and was able to reduce this cost significantly.

\section{Ideas for the paper: divisor sums}

We note that the ultimate goal here would be to show that $\Lambda$ acts like $\lambda'$ with little error.  If we could do this, we would study primes by analyzing a function $\lambda'$ that looks like the binary divisor function $\tau$.  Sums of $\tau$ in arithmetic progressions are well-understood, with the classical result in this vein being proven in the 1950s by Selberg and Hooley - see \cite{Ho}, \cite[p. 234-237]{Se}, \cite[Corollary 1]{HB79}.
\begin{theorem}\label{binarydivisor}
If $q\leq x^{\frac 23-2\varepsilon}$ and $(a,q)=1$, then
$$\sum_{\substack{n\leq x\\ n\equiv a\pmod q}}\tau(n)=\frac{1}{\phi(q)}\sum_{\substack{n\leq x\\ (n,q)=1}}\tau(n)+O\left(\frac{x}{q^{1+\varepsilon}}\right).$$
\end{theorem}
We note that this theorem also applies to twisted divisor functions, as well as divisor sums over subdyadic intervals as follows.
\begin{theorem}\label{binaryTdivisor}
Let $q\leq x^{\frac 23-2\varepsilon}$ and $(a,q)=1$.  Define $f$ such that either $f(n)=\chi(n)$ or $f(n)=1$.  Let $U,V>x^\varepsilon$ for some small $\varepsilon>0$, and let $\mathcal U=[U,U+U']$ and $\mathcal V=[V,V+V']$ where $\frac{U}{\log^B U}\leq U'\leq U$ and $\frac{V}{\log^B V}\leq V'\leq V$ for some large fixed $B$.  Then
$$\mathop{\sum\sum}\limits_{\substack{u\in U,v\in V \\ uv\equiv a\pmod q}}f(v)=\frac{1}{\phi(q)}\mathop{\sum\sum}\limits_{\substack{u\in U,v\in V \\ (uv,q)=1}}f(v)+O\left(\frac{x}{q^{1+\varepsilon}}\right).$$
\end{theorem}
The case where $f=1$ is simply a subcase of Theorem \ref{binarydivisor}.  The case where $f=\chi$ appears in several places (e.g. \cite{Sa} and \cite{WrLM}).

The first result on ternary sums on arithmetic progressions that moved beyond the square-root barrier came from Friedlander and Iwaniec in 1985 \cite[p. 339]{FI85}.
\begin{theorem}[Friedlander-Iwaniec, 1985]\label{FIternary}
If $q<x^{\frac{58}{115}-\varepsilon}$ and $(a,q)=1$ then
$$\sum_{\substack{n\leq x\\ n\equiv a\pmod q}}\tau_3(n)=\frac{1}{\phi(q)}\sum_{\substack{n\leq x\\ (n,q)=1}}\tau_3(n)+O\left(\frac{x}{q^{1+\varepsilon}}\right).$$
\end{theorem}
Since $\tau(n)=(1\ast 1)(n)$ and $\tau_3(n)=(1\ast 1\ast 1)(n)$, it seems logical to hope that we can generalize these results to other convolutions such as $(\chi \ast 1)(n)$ and $(\chi \ast 1\ast 1)(n)$.  A key step in \cite{FI03} is the realization that this can indeed happen.  In fact, equation (5.6) in \cite{FI03} gives the following.
\begin{theorem}[Friedlander-Iwaniec, 2003]\label{FIternary2}
If $x^{\frac{92}{185}}<q<x^{\frac{58}{115}-\varepsilon}$ and $(a,q)=1$, then
$$\sum_{\substack{dm\leq x \\ d>D^*\\dm\equiv a\pmod q}}\lambda(d)=\frac{1}{\phi(q)}\sum_{\substack{dm\leq x \\ d>D^*\\ (dm,q)=1}}\lambda(d)+O\left(Dx^{\frac{271}{300}+\varepsilon}q^{-97/120}\right).$$
\end{theorem}
This sum provides the bound for $\theta$ in \cite{FI03}, as all of the other sums in that paper give a looser bound than this for $q$.  Thus, we see that if one is to improve the level of distribution (i.e. increase the value of $\theta$ such that the theorem holds for $q<x^\theta$), one must somehow improve this ternary divisor sum result, either through better Kloosterman sum bounds or by reducing to a binary divisor sum.  

In this paper, we use both techniques.  Ultimately, we prove the following.
\begin{theorem}\label{MainSubThm1}
If $x^{\frac 12}<D^3q<x^{\frac{30}{59}-2\alpha}$ for some small positive $\alpha<1/200$, and if $(a,q)=1$, then
$$\sum_{\substack{dm\leq x \\ d>D^*\\dm\equiv a\pmod q}}\lambda(d)=\frac{1}{\phi(q)}\sum_{\substack{dm\leq x \\ d>D^*\\ (dm,q)=1}}\lambda(d)+O\left(\frac{x}{q^{1+\frac \alpha 2}}\right).$$
\end{theorem}
\begin{theorem}\label{MainSubThm2}
If $x^{\frac 12}<D^3Q<x^{\frac{16}{31}-2\alpha}$ and $(a,q)=1$, then there exists a $\varepsilon>0$ such that for all but $Q^{1-2\alpha-o(1)}$ of the $q\in [Q,2Q]$,
$$\max_{(a,q)=1}\left|\sum_{\substack{dm\leq x \\ d>D^*\\dm\equiv a\pmod q}}\lambda(d)-\frac{1}{\phi(q)}\sum_{\substack{dm\leq x \\ d>D^*\\ (dm,q)=1}}\lambda(d)\right|\ll \frac{x}{q^{1+\frac \alpha 2}}.$$
\end{theorem}

One can then simply plug this result into \cite[Lemma 3.2]{Sa} to prove the theorem.

\section{New Kloosterman ideas}

The trick for proving Theorems \ref{MainSubThm1} and \ref{MainSubThm2} will be to apply a relatively new result of Shparlinski \cite{Shp} on trilinear Kloosterman sums.  Let us denote
$$K_\chi=K_\chi(q,a)=\sum_{0<h\leq H}\sum_{0\leq m\leq M}\sum_{0\leq n\leq N}f_1(m)f_2(n)e\left(\frac{ah\bar m\bar n}{q}\right),$$
where one of the $f_j$ is $\chi$ and the other is 1.  In the original paper of Friedlander and Iwaniec \cite{FI85} on ternary divisor sums, they used the following inequality, which appears as (2.6) in that work.  
\begin{theorem}[Friedlander-Iwaniec, 1985]\label{FI:2.6}
Let $(a,q)=1$.  Then
\begin{align*}K_\chi(a)
\ll & Dx^\varepsilon \left(q^\frac 34 H^{\frac 12}M^{\frac 12} + q^\frac 14HM + q^\frac 13H^\frac 23M^\frac 13N^\frac 23+HM^\frac 23N^\frac 23+q^{-1}HMN\right).
\end{align*}
\end{theorem}
We will apply this in the case where at least one of $m$, $n$ is fairly large.

Note that the authors originally proved this in \cite{FI85} without the character, and the bound that they found did not have the additional $D$ in the front.  The later adaptation of this result to the case where one of the variables is twisted by $\chi$ costs an additional factor of $D$ (as noted in the discussion just above (5.6) in \cite{FI03})

In the case where both $m$ and $n$ are fairly small, we will use the aforementioned results of Shparlinski \cite{Shp}, which are as follows.  For a triple of 1-bounded functions $g=(g_1,g_2,g_3)$, define
$$K_g=K_g(q,a)=\sum_{0<h\leq H}\sum_{M\leq m\leq 2M}\sum_{N\leq n\leq 2N}g_1(h)g_2(m)g_3(n)e\left(\frac{ah\bar m\bar n}{q}\right).$$
Theorem 1.1 of \cite{Shp} gives a bound for individual $q$.
\begin{theorem}[Shparlinski, 2019]\label{Shp:T1} Let $(a,q)=1$.  Then
$$K_g(q,a)\ll \left(HM+(HM)^\frac 34Q^\frac 14\right)\left(N^\frac 78Q^{-\frac 18}+N^\frac 12\right)Q^{o(1)}.$$
\end{theorem}
Meanwhile, Theorem 1.2 of that paper gives a bound for a range of $q$.
\begin{theorem}[Shparlinski, 2019]\label{Shp:T2}
Let $\kappa>0$ be a fixed real number, and let $Q$ be sufficiently large.  For all but at most $Q^{1-4\kappa+o(1)}$ values of $q\in [Q,2Q]$,
\begin{gather}\label{goodq}\max_{(a,q)=1}|K_g(q,a)|\ll \left(HM+(HM)^\frac 34Q^\frac 14\right)\left(NQ^{-\frac 14}+N^\frac 12\right)Q^{\kappa+o(1)}.
\end{gather}
\end{theorem}
It is the application of these Shparlinski results that gives us savings over the original \cite{FI03} paper and allows for our new result.

\section{Remarks}\label{remarks}

As noted above, the original Friedlander-Iwaniec idea behind these methods is to turn the function $\Lambda$ into a function that behaves like the $k$-fold divisor function $\tau_k$.  The ultimate goal would be to show that $\Lambda$ is very close to $\lambda'$, which acts like the binary divisor function $\tau$.  Since $\tau$ is equidistributed modulo $q$ for $q<x^{\frac 23-\varepsilon}$, it stands to reason that these methods might be used to find the distribution of primes modulo $q$ over a similar range of $q$ if one assumes Siegel zeroes.  Indeed, Theorems 2.2 and 2.3 of \cite{WrLM} prove that if $q<x^{\frac 23-\varepsilon}$, one has the correct upper bound for $\psi(x,q,a)$ (with $o(\frac{\pi(x)}{\phi(q)})$ error), and for a fixed $a$, one can prove that $$\psi(x,q,a)=\frac{1+o(1)}{\phi(q)}\pi(x)$$for most $q\sim Q$ if $Q<x^{\frac 23-\varepsilon}$.


We also note here that Theorem \ref{FIternary} is not the optimal known result for ternary divisor sums.  Indeed, Heath-Brown \cite{HB86} improved the exponent to $\frac 12+\frac{1}{82}$ for individual $q$ and $\frac 12+\frac{1}{42}$ over a range of $q$. Fouvry, Kowalski, and Michel \cite{FKM} later proved exponents of $\frac 12+\frac{1}{46}$ in the case that $q$ is prime and $\frac 12+\frac{1}{34}$ when $q$ is averaged over a fixed residue class, while Sharma \cite{Sha} raised the exponent to $\frac 12+\frac{1}{30}$ for individual $q$ in the case where $q$ is square-free or an odd prime power.  However, it is not yet known how to adapt \cite{HB86}, \cite{FKM}, or \cite{Sha} to more general divisor sums like ours.

By contrast, \cite{FI85} and \cite{Shp} adapt more easily to the introduction of $\chi$.  Applying $\chi$ to \cite{FI85} only costs us an additional multiple of $D$ on the error term (as noted above), while \cite{Shp} is actually stated such that one could insert any 1-bounded function into the sum with no change in the bound.

\section{Application to the Sachpazis paper}

Before we begin the proofs of Theorems \ref{MainSubThm1} and \ref{MainSubThm2}, it is worth mentioning how these results apply in the aforementioned paper of Sachpazis.  For $x=D^V$, let
$$z=D^{\min\{\sqrt{V/(\log \eta)},2\}}.$$
In (2.7) of that paper, the author finds that for any $q\leq x^{\frac 23-\varepsilon}$,
\begin{align}\label{tocon}
\begin{split}
\psi(x;q,a)-\frac{\psi(x)}{\phi(q)}(1-\mathds{1}_{D\mid q}\chi(a))=&\sum_{\substack{n\leq x\\n\equiv a\mod{q}\\P^-(n)>z}}\lambda'(n)-\frac{1-\mathds{1}_{D\mid q}\chi(a)}{\phi(q)}\sum_{\substack{n\leq x\\(n,q)=1\\P^-(n)>z}}\lambda'(n)\\
&+\frac{1-\mathds{1}_{D\mid q}\chi(a)}{\phi(q)}\sum_{\substack{k\ell\leq x,\,k>z\\(k\ell,q)=1\\P^-(k\ell)>z}}\lambda(k)\Lambda(\ell)\\
&-\sum_{\substack{k\ell\leq x,\,k>z\\k\ell\equiv a\mod{q}\\P^-(k\ell)>z}}\lambda(k)\Lambda(\ell)+O\bigg(\frac{z\log x}{\log z}\bigg).
\end{split}
\end{align}
The author then finds that 
\begin{align*}
\left|\sum_{\substack{n\leq x\\n\equiv a\mod{q}\\P^-(n)>z}}\lambda'(n)-\frac{1-\mathds{1}_{D\mid q}\chi(a)}{\phi(q)}\sum_{\substack{n\leq x\\(n,q)=1\\P^-(n)>z}}\lambda'(n)\right|\ll \frac{x}{\phi(q)}\left(\frac{V^{12}}{\eta}+\exp\left(-c_\varepsilon\sqrt{V\log \eta}\right)\right),
\end{align*}
and
$$\sum_{\substack{k\ell\leq x,\,k>z\\(k\ell,q)=1\\P^-(k\ell)>z}}\lambda(k)\Lambda(\ell)\ll \frac{xV^4}{\eta}+x\exp\Big(-c'\sqrt{V\log\eta}\Big).$$
For the remaining sum, Section 6 of that paper finds
\begin{align}\label{S2tfb}
\begin{split}
\sum_{\substack{k\ell\leq x,\,k>z\\k\ell\equiv a\mod{q}\\P^-(k\ell)>z}}\lambda(k)\Lambda(\ell)&\ll_{\varepsilon}\frac{x}{\phi(q)}\bigg(\frac{V^{16}}{\eta} +\exp\Big(-c_{\varepsilon}'\sqrt{V\log\eta}\Big)\bigg)+(\log x)\mathcal C,    
\end{split}
\end{align}
where
\begin{align*}
\mathcal C=\sum_{\substack{d_1,d_2,d_3\leq x^{\varepsilon/200}\\(d_1d_2d_3,q)=1}}w(d_1)\chi(d_1)w(d_2)w(d_3)\bigg(\sum_{\substack{k\ell\leq x/(d_1d_2d_3)\\k\ell\equiv a\bar{d_1d_2d_3}\mod{q}}}\lambda(k)-\frac{1}{\phi(q)}\!\!\sum_{\substack{k\ell\leq x/(d_1d_2d_3)\\(k\ell,q)=1}}\lambda(k)\bigg).
\end{align*}
From this, we can see that the allowable range of $q$ is determined by $\mathcal C$.  Theorem \ref{FIternary2} then allows for $q\leq x^{\frac{58}{115}-\varepsilon}$, whereas our new Theorems \ref{MainSubThm1} and \ref{MainSubThm2} allow for $q\leq x^{\frac{30}{59}-\varepsilon}$ and $Q\leq x^{\frac{16}{31}-\varepsilon}$, respectively.

\section{Partitioning the sum}

To prove Theorems \ref{MainSubThm1} and \ref{MainSubThm2}, it will be helpful to first turn
$$\sum_{\substack{dm\leq x \\ d>D^*\\dm\equiv a\pmod q}}\lambda(d)$$
into a ternary sum
$$\sum_{\substack{dm\leq x \\ d>D^*\\dm\equiv a\pmod q}}\lambda(d)=\mathop{\sum\sum\sum}\limits_{\substack{uvm\leq x \\ uv>D^* \\ muv\equiv a\pmod q }}\chi(v).$$
Our goal will be to show that 
\begin{gather}\label{goal:ternary}\mathop{\sum\sum\sum}\limits_{\substack{uvm\leq x \\ uv>D^* \\ muv\equiv a\pmod q }}\chi(v)=\frac{1+O(q^{-\alpha})}{\phi(q)}\mathop{\sum\sum\sum}\limits_{\substack{uvm\leq x \\ uv>D^* \\ muv\equiv a\pmod q }}\chi(v)\end{gather}
for some $\alpha>0$. 

We note that over the regions where any of $u$, $v$, or $m$ is larger than $Dq^{1+\alpha}$, the bound in (\ref{goal:ternary}) holds trivially, since one can simply apply the congruence condition to the large variable.  In other words,
$$\mathop{\sum\sum\sum}\limits_{\substack{uvm\leq x \\ uv>D^* \\ muv\equiv a\pmod q }}\chi(v)=\frac{1+O(q^{-\alpha})}{\phi(q)}\mathop{\sum\sum\sum}\limits_{\substack{uvm\leq x \\ uv>D^* \\ \max\{u,v,m\}>Dq^{1+\alpha} \\ muv\equiv a\pmod q }}\chi(v).$$
In the remaining intervals, the requirement that $uv>D^*$ is now redundant and can be ignored.

Similarly, if, say, $u<\frac{x}{q^{\frac 32+\alpha}}$ then $mv>q^{\frac 32+\alpha}$ and hence we can apply Theorem \ref{binaryTdivisor} to the sum over $v$ and $m$.  So
$$\mathop{\sum\sum\sum}\limits_{\substack{uvm\leq x \\ uv>D^* \\ muv\equiv a\pmod q }}\chi(v)=\frac{1+O(q^{-\alpha})}{\phi(q)}\mathop{\sum\sum\sum}\limits_{\substack{uvm\leq x \\ \min\{u,v,m\}<\frac{x}{q^{\frac 32+\alpha}} \\ \max\{u,v,m\}\leq Dq^{1+\alpha} \\ muv\equiv a\pmod q }}\chi(v).$$
Thus, we only need to consider the intervals where
$$u,v,m\in \bigg[\frac{x}{q^{\frac 32+\alpha}},Dq^{1+\alpha}\bigg).$$  
Define this interval to be $\mathcal I$.

For the remaining intervals, we partition the sum into subdyadic intervals, and we change the notation from denoting which variable has the character to denoting the size of the variables.  First, we break the interval $\mathcal I$ into dyadic intervals $[U,U(1+\zeta))$ for some fixed $0<\zeta\leq 1$, and then we break the dyadic intervals into subintervals $[M,M(1+\Delta))$ with $\Delta \ll \frac{1}{\log^8 x}$.    So we have triples of intervals $\mathcal M_1,\mathcal M_2, \mathcal M_3$ where each of the $\mathcal M_i=[M_i,M_i(1+\Delta))$ with  $M_1$, $M_2$, and $M_3$ such that $M_1M_2M_3\ll x$.  

Define $\mathcal M=(\mathcal M_1,\mathcal M_2, \mathcal M_3)$ to be such a triple, and define $\mathcal J$ to be the set of these triples of intervals.  From here, we will write
$$\mathop{\sum\sum\sum}\limits_{\substack{t_1,t_2,t_3\\ t_i \in \mathcal M_i\\ t_1t_2t_3\equiv a\pmod q }}f_1(t_1)f_2(t_2)f_3(t_3),$$
where one of the $f_i(t_i)=\chi(t_i)$ and the other $f_i(t_i)=1$.

Note that
$$\mathop{\sum\sum\sum}\limits_{\substack{(\mathcal M_1,\mathcal M_2, \mathcal M_3)\in \mathcal M \\ M_1M_2M_3\leq \frac{x}{\log^{12} x }}}\mathop{\sum\sum\sum}\limits_{\substack{t_1,t_2,t_3\\ t_i \in \mathcal M_i\\ t_1t_2t_3\equiv a\pmod q }}f_1(t_1)f_2(t_2)f_3(t_3)\ll \sum_{n\leq \frac{8x}{\log^{12} x }}\tau_3(n)\ll \frac{x}{\log^{10} x} .$$
Moreover, 
$$\mathop{\sum\sum\sum}\limits_{\substack{(\mathcal M_1,\mathcal M_2, \mathcal M_3)\in \mathcal M \\ M_1M_2M_3\leq \frac{x}{\log^{12} x }}}\mathop{\sum\sum\sum}\limits_{\substack{t_1,t_2,t_3 \\ t_i \in \mathcal M_i\\ t_1t_2t_3\equiv a\pmod q }}f_1(t_1)f_2(t_2)f_3(t_3)\ll \sum_{x<n\leq x+\frac{8x}{\log^{8} x }}\tau_3(n)\ll \frac{x}{\log^{6} x} .$$
So
$$\mathop{\sum\sum\sum}\limits_{\substack{uvm\leq x \\ u,v,m\in \mathcal I \\ muv\equiv a\pmod q }}\chi(v)=\mathop{\sum\sum\sum}\limits_{\substack{(\mathcal M_1,\mathcal M_2, \mathcal M_3)\in \mathcal M \\ \frac{x}{\log^{12} x}<M_1M_2M_3\leq x }}\mathop{\sum\sum\sum}\limits_{\substack{t_1,t_2,t_3 \\ t_i \in \mathcal M_i\\ t_1t_2t_3\equiv a\pmod q }}f_1(t_1)f_2(t_2)f_3(t_3)+O\left( \frac{x}{\log^{6} x}\right).$$
Without loss of generality, we will assume that
\begin{gather}\label{boundsform1m2m3}
\frac{x}{q^{\frac 32+\alpha}}\leq M_3\leq M_2\leq M_1\leq Dq^{1+\alpha}.
\end{gather}
For ease of notation, define
$$\mathcal D_{\mathcal M,f}=\mathop{\sum\sum\sum}\limits_{\substack{t_1,t_2,t_3 \\  t_i \in \mathcal M_i\\ t_1t_2t_3\equiv a\pmod q }}f_1(t_1)f_2(t_2)f_3(t_3),$$
where the $f$ denotes the order of the choices of $f_i$, and define
$$\mathcal D^*_{\mathcal M,f}=\mathop{\sum\sum\sum}\limits_{\substack{t_1,t_2,t_3 \\  t_1t_2t_3>x\\ t_i \in \mathcal M_i\\ (t_1t_2t_3,q)=1 }}f_1(t_1)f_2(t_2)f_3(t_3),$$
Our goal here is to bound 
$$\left|\mathcal D_{\mathcal M,f}-\frac{1}{\phi(q)}\mathcal D^*_{\mathcal M,f}\right|.$$
However, since one of the functions is $\chi$ and $\sum_{A\leq n\leq B}\chi(n)\ll D$ for any $A$ and $B$, we can see that trivially,
$$\left|\frac{1}{\phi(q)}\mathcal D^*_{\mathcal M,f}\right|\ll \frac{DM_1M_2}{\phi(q)},$$
which is much smaller than $\frac{x}{q^{1+\alpha}}$ for small values of $\alpha$.  Hence, it will suffice to bound $\left|\mathcal D_{\mathcal M,f}\right|.$

\section{The Friedlander-Iwaniec Bounds}
As we mentioned in Section \ref{remarks}, some of the ternary sum bounds in the literature are useful for the evaluation of $\mathcal D$ above, but others are not.  We will eventually need to show that the Shparlinski bounds can in fact be applied to a Kloosterman sum that arises from the estimation of $\left|\mathcal D_{\mathcal M,f}\right|$.

First, however, we summarize the results from Friedlander and Iwaniec \cite{FI03}.

\begin{lemma}  Let $\alpha$ be such that $0<\alpha<1/100$, and let $x^\alpha\leq q$.  Then for any small $\varepsilon>0$,
\begin{align*}\left|\mathcal D_{\mathcal M,f}\right|\ll & \left(\frac{x}{q^{1+\alpha}}+Dq^\frac 12\right)q^\varepsilon\\
&+D^2q^\alpha x^\varepsilon \left(q^\frac 14 M_1^{\frac 12}M_3^{\frac 12} + q^\frac 14M_3 + M_1^\frac 13M_3^\frac 13M_2^\frac 23+M_3^\frac 23M_2^\frac 23+q^{-1}M_2M_3\right).
\end{align*}
\end{lemma}
\begin{proof}
For a given $A$ and $B$, define
$$c_q(h)=\frac 1{q}\int_{A}^{B}e\left(\frac{hz}{q}\right)dz.$$
By (3.3) of \cite{FI85}, we have that for any $H$ with $1\leq H<q$,
\begin{align}\label{3.3}\left|\mathop{\sum\sum\sum}\limits_{\substack{t_1,t_2,t_3\\ t_i \in \mathcal M_i\\ t_1t_2t_3\equiv a\pmod q }}1-\frac{1}{\phi(q)}\mathop{\sum\sum\sum}\limits_{\substack{t_1,t_2,t_3\\ t_i \in \mathcal M_i\\ (t_1t_2t_3,q)=1 }}1\right|
\ll \left(\frac{M_2M_3}{H}+q^\frac 12+G(H,\mathcal M_2,M_3)\right)q^\varepsilon,
\end{align}
where
$$G(H,\mathcal M_2,M_3)=\max_{\substack{a\\(a,q)=1}}\left|\sum_{1\leq |h|\leq H}c(h)\mathop{\sum\sum}\limits_{t_2\in \mathcal M_2,t_3\in \mathcal M_3}\mathop{\sum\sum}\limits_{t_2\in \mathcal M_2,t_3\in \mathcal M_3}e\left(-\frac{a\overline{t_2} \overline{t_3} h}{q}\right)\right|.$$
We take $$H=\frac{Dq^{1+\alpha}M_2M_3}{x}.$$
Since $M_1\leq Dq^{1+\alpha}$, our choice of $H$ is clearly greater than 1.  We can also say that
\begin{gather}\label{HM1}
H\ll \frac{Dq^{1+\alpha}}{M_1}.
\end{gather}
From the proof of Proposition 2 in \cite{FI85}, the authors note that by Abel's partial summation,
\begin{align*}&\left|\sum_{1\leq |h|\leq H}c(h)\mathop{\sum\sum}\limits_{t_2\in \mathcal M_2,t_3\in \mathcal M_3}\mathop{\sum\sum}\limits_{t_2\in \mathcal M_2,t_3\in \mathcal M_3}e\left(-\frac{a\overline{t_2} \overline{t_3} h}{q}\right)\right|\\
&\ll q^{-1}M_1\left|\sum_{1\leq |h|\leq H}\mathop{\sum\sum}\limits_{t_2\in \mathcal M_2,t_3\in \mathcal M_3}e\left(-\frac{a\overline{t_2} \overline{t_3} h}{q}\right)\right|+Hq^{-2}M_1^2\max_{1\leq y\leq H}\left|\sum_{1\leq |h|\leq y}\mathop{\sum\sum}\limits_{t_2\in \mathcal M_2,t_3\in \mathcal M_3}e\left(-\frac{a\overline{t_2} \overline{t_3} h}{q}\right)\right|\\
&\ll Dq^{-1+\alpha}M_1\max_{1\leq y\leq H}\left|\sum_{1\leq |h|\leq y}\mathop{\sum\sum}\limits_{t_2\in \mathcal M_2,t_3\in \mathcal M_3}e\left(-\frac{a\overline{t_2} \overline{t_3} h}{q}\right)\right|,
\end{align*}
by (\ref{HM1}).

By Theorem \ref{FI:2.6}, we can then bound the expression above, taking $H$ as defined, $M=M_3$, and $N=M_2$.  This gives
\begin{align*}
G(H,\mathcal M_2,M_3)\ll & Dq^{-1+\alpha}M_1x^\varepsilon \left(q^\frac 34 H^{\frac 12}M_3^{\frac 12} + q^\frac 14HM_3 + q^\frac 13H^\frac 23M_3^\frac 13M_2^\frac 23+HM_3^\frac 23M_2^\frac 23+q^{-1}HM_2M_3\right).
\end{align*}
Invoking (\ref{HM1}) again, we have
\begin{align}\label{Gbegin}
G(H,\mathcal M_2,M_3)\ll & Dx^\varepsilon q^\alpha\left(q^\frac 14 M_1^{\frac 12}M_3^{\frac 12} + q^\frac 14M_3 + M_1^\frac 13M_3^\frac 13M_2^\frac 23+M_3^\frac 23M_2^\frac 23+q^{-1}M_2M_3\right).
\end{align}
Hence
\begin{align*}&\left|\mathop{\sum\sum\sum}\limits_{\substack{t_1,t_2,t_3\\ t_i \in \mathcal M_i\\ t_1t_2t_3\equiv a\pmod q }}1-\frac{1}{\phi(q)}\mathop{\sum\sum\sum}\limits_{\substack{t_1,t_2,t_3\\ t_i \in \mathcal M_i\\ (t_1t_2t_3,q)=1 }}1\right|\\
&\ll \frac{x}{Dq^{1+\alpha}}+q^{\frac 12+\varepsilon}+Dq^\alpha x^\varepsilon \left(q^\frac 14 M_1^{\frac 12}M_3^{\frac 12} + q^\frac 14M_3 + M_1^\frac 13M_3^\frac 13M_2^\frac 23+M_3^\frac 23M_2^\frac 23+q^{-1}M_2M_3\right).
\end{align*}
By the observations in the adaptation cited in the proof in \cite[Section 5]{FI03}, we can apply this bound to $$\left|\mathcal D_{\mathcal M,f}-\frac{1}{\phi(q)}\mathcal D^*_{\mathcal M,f}\right|$$
as well with only the gain of an additional factor of $D$. The lemma then follows.
\end{proof}
We can simplify this bound as follows.
\begin{lemma}\label{G:FI} For $0<\alpha<1/100$, let $x^\frac 12\leq D^3q<x^{\frac{8}{15}-\alpha}$.  Then for any $\varepsilon$ with $0<\varepsilon<\alpha/2$,
$$\left|\mathcal D_{\mathcal M,f}\right|\ll D^3x^\varepsilon q^{\frac 14+\alpha} \left(\frac{x}{M_2}\right)^{\frac 12}+\frac{x}{q^{1+\alpha}}.$$
\end{lemma}
\begin{proof}
In this case,
$$\frac{x}{q^{1+\alpha}}\gg q^{\frac 12+\varepsilon}.$$
So we need only handle the long parenthesized section of the bound.


Using (\ref{HM1}) and the bounds that $M_3\leq x^\frac 13$, $M_2M_3\leq x^\frac 23$, and $M_2\leq \sqrt{M_1M_2}\leq q^\frac 34$, we have
\begin{align*}
D^2q^\alpha & x^\varepsilon \left(q^\frac 14 M_1^{\frac 12}M_3^{\frac 12} + q^\frac 14M_3 + M_1^\frac 13M_3^\frac 13M_2^\frac 23+M_3^\frac 23M_2^\frac 23+q^{-1}M_2M_3\right)\\
\ll & D^3x^\varepsilon q^\alpha\left(q^\frac 14 \left(\frac{x}{M_2}\right)^{\frac 12} + q^\frac 14x^\frac 13 + x^\frac 13q^\frac 14+x^\frac 49+q^{-1}x^\frac 23\right).
\end{align*}
All of the terms inside the parentheses except possibly the first term are clearly less than $\frac{x}{D^3q^{1+3\alpha}}$ as long as $D^3q<x^{\frac{8}{15}-5\alpha}$.  Since we can take any small $\varepsilon$, we will assume $\varepsilon<\frac {\alpha}{2}$ as in the statement of the lemma.  Hence, the above is
\begin{align*}
\ll & D^3x^\varepsilon q^{\frac 14 +\alpha}x^\frac 12M_2^{-\frac 12}+\frac{x}{q^{1+\alpha}}.
\end{align*}
\end{proof}

\section{Prelude to the Shparlinski results}

It remains now to show that Theorems \ref{Shp:T1} and \ref{Shp:T2} can also be used to bound $|\mathcal D_{\mathcal M,f}|$.

To begin, we work through the steps of Section 3 of \cite{FI85}.  By (3.2) of \cite{FI85}, for any $q$ and any $H$ with $0<H<q$,
\begin{gather}\label{congclass01}\sum_{\substack{A\leq m\leq B \\ m\equiv a\bar r \pmod{q}}}1=\frac{B-A}{q}+\sum_{0<|h|\leq H}c_q(h)e\left(-\frac{a\bar r h}{q}\right)+O\left(\varrho\left(\frac{B-a\bar r}{q}\right)+\varrho\left(\frac{A-a\bar r}{q}\right)\right),\end{gather}
where for any real number $z$,
$$\varrho(z)=\min(1,(H||z||)^{-1})$$
and $||\cdot ||$ indicates distance to the nearest integer.  Lemma 3.1 of \cite{FI85} states that for any $V\geq 1$ and any $\varepsilon>0$,
\begin{gather}\label{boundforvarrho}\sum_{\substack{K'\leq k\leq K'+K \\ (k,q)=1}}\varrho\left(\frac{V-a\bar k}{q}\right)\ll \left(\frac{C'}{H}+q^\frac 12+\frac C{q}\right)x^{\varepsilon}.
\end{gather}
This of course means that for double sums, we have a similar bound, as the nonnegativity of $\varrho$ gives
\begin{align*}\mathop{\sum\sum}\limits_{\substack{J'\leq j\leq J'+J\\K'\leq k\leq K'+K  \\ (jk,q)=1}}\varrho\left(\frac{V-a\bar j\bar k}{q}\right)\ll & \sum_{J'K'\leq n\leq J'K'+2J'K+2JK'}\tau(n)\varrho\left(\frac{V-a\bar n}{q}\right)\\
\ll &x^\varepsilon\sum_{J'K'\leq n\leq J'K'+2J'K+2JK'}\varrho\left(\frac{V-a\bar n}{q}\right)\\
\ll &\left(\frac{J'K'}{H}+q^\frac 12+\frac{J'K'}{q}\right)x^{\varepsilon}.
\end{align*}
Note that this bound also holds for 
\begin{align*}\mathop{\sum\sum}\limits_{\substack{J'\leq j\leq J'+J\\K'\leq k\leq K'+K  \\ (jk,q)=1}}g_1(j)g_2(j)\varrho\left(\frac{V-a\bar j\bar k}{q}\right)
\end{align*}
for 1-bounded functions $g_1$ and $g_2$.

We then prove the following.
\begin{lemma}\label{ternary} Let $0<\alpha<1/100$, and let $H=\frac{Dq^{1+\alpha}M_2M_3}{x}$.  Then for any $\varepsilon$ with $0<\varepsilon<\alpha/4$,
$$\mathop{\sum\sum\sum}\limits_{\substack{t_j\in \mathcal M_j\\ t_1t_2t_3\equiv a\pmod q }}\chi(t_3)=\mathop{\sum\sum}\limits_{t_2\in \mathcal M_2,t_3\in \mathcal M_3}\chi(t_3)\sum_{0<|h|\leq H}c_q(h)e\left(-\frac{a\overline{t_2} \overline{t_3} h}{q}\right)+O\left(\left(q^{\frac 12+\alpha}+\frac{x}{q^{1+\alpha}}\right)x^\varepsilon\right),$$
$$\mathop{\sum\sum\sum}\limits_{\substack{t_j\in \mathcal M_j\\ t_1t_2t_3\equiv a\pmod q }}\chi(t_2)=\mathop{\sum\sum}\limits_{t_2\in \mathcal M_2,t_3\in \mathcal M_3}\chi(t_2)\sum_{0<|h|\leq H}c_q(h)e\left(-\frac{a\overline{t_2} \overline{t_3} h}{q}\right)+O\left(\left(q^{\frac 12+\alpha}+\frac{x}{q^{1+\alpha}}\right)x^\varepsilon\right),$$
and
\begin{align*}&\mathop{\sum\sum}\limits_{t_2,t_3\in \mathcal M_2,\mathcal M_3}\sum_{\substack{M_1\leq t_1\leq M_1+M_1' \\ t_1\equiv a\bar t_2\bar t_3 \pmod{q}\\ t_1\equiv v\pmod D}}\chi(t_1)\\
&=\sum_{0<v<D}\frac{\chi(D)}{\phi(J)}\sum_{\chi' \pmod J}\mathop{\sum\sum}\limits_{\substack{t_2,t_3\in \mathcal M_2,\mathcal M_3 } }\chi'(t_2)\chi'(t_3)\overline{\chi'(v)}\sum_{0<|h|\leq H}c_q(h)e\left(-\frac{a\bar t_2\bar t_3h}{q}\right)e\left(-\frac{vh}{D}\right)\\
&\phantom{=}+O\left(\left(D^\frac 32q^{\frac 12+\alpha}+\frac{x}{q^{1+\alpha}}\right)x^\varepsilon\right).
\end{align*}
In particular, for all of these cases, if $D^2q<x^{\frac 23-4\alpha}$, then there exist 1-bounded functions $g_1$, $g_2$, and $g_3$ such that
\begin{gather}\label{chbound}\left|\mathcal D_{\mathcal M,f}\right|\ll D\left|\mathop{\sum\sum}\limits_{\substack{t_2,t_3\in \mathcal M_2,\mathcal M_3 \\ (t_2t_3,q)=1 }}\sum_{0<|h|\leq H}g_1(h)g_2(t_2)g_3(t_3)c_q(h)e\left(-\frac{a\overline{t_2} \overline{t_3} h}{q}\right)\right|+\frac{x}{q^{1+\frac \alpha 2}}.
\end{gather}

\end{lemma}

\begin{proof}
We begin with the sum where $t_3$ is twisted by a character.  Recalling equation (\ref{congclass01}):
\begin{align*}&\mathop{\sum\sum}\limits_{t_2,t_3\in \mathcal M_2,\mathcal M_3}\sum_{\substack{M_1\leq t_1\leq M_1+M_1' \\ t_1\equiv a\bar t_2\bar t_3 \pmod{q}}}\chi(t_3)\\
&=\mathop{\sum\sum}\limits_{t_2,t_3\in \mathcal M_2,\mathcal M_3}\chi(t_3)\left[\frac{M_1'}{q}+\sum_{0<|h|\leq H}c_q(h)e\left(-\frac{a\bar t_2\bar t_3 h}{q}\right)+O\left(\varrho\left(\frac{M_1+M_1'-a\bar t_2\bar t_3}{q}\right)+\varrho\left(\frac{M_1-a\bar t_2\bar t_3}{q}\right)\right)\right]\\
&=O\left(\frac{M_1M_2D}{q}\right)+\mathop{\sum\sum}\limits_{t_2,t_3\in \mathcal M_2,\mathcal M_3}\chi(t_3)\sum_{0<|h|\leq H}c_q(h)e\left(-\frac{a\bar t_2\bar t_3 h}{q}\right)+O\left(\left(\frac{M_2M_3}{H}+q^\frac 12+\frac{M_2M_3}{q}\right)x^\varepsilon\right).
\end{align*}
By definition of $H$, $\frac{M_2M_3}{H}\ll \frac{x}{Dq^{1+\alpha}}$.  So $q^{\frac 12+\alpha}x^\varepsilon+\frac{x^{1+\varepsilon}}{Dq^{1+\alpha}}$ dominates the big-O terms, and hence
\begin{align*}&\mathop{\sum\sum}\limits_{t_2,t_3\in \mathcal M_2,\mathcal M_3}\sum_{\substack{M_1\leq t_1\leq M_1+M_1' \\ t_1\equiv a\bar t_2\bar t_3 \pmod{q}}}\chi(t_3)=\mathop{\sum\sum}\limits_{t_2,t_3\in \mathcal M_2,\mathcal M_3}\chi(t_3)\sum_{0<|h|\leq H}c_q(h)e\left(-\frac{a\bar t_2\bar t_3 h}{q}\right)+O\left(\left(q^{\frac 12+\alpha}+\frac{x}{q^{1+\alpha}}\right)x^\varepsilon\right).
\end{align*}
The proof for the sum where $t_2$ is twisted by a character is nearly identical.

For the sum where $t_1$ is twisted by a character,
$$\mathop{\sum\sum\sum}\limits_{\substack{t_j\in \mathcal M_j\\ t_1t_2t_3\equiv a\pmod q }}\chi(t_1)=\sum_{0<v<D}\chi(v)\mathop{\sum\sum}\limits_{\substack{t_2,t_3\in \mathcal M_2,\mathcal M_3 \\ (t_2t_3,q)=1 }}\mathop{\sum}\limits_{\substack{t_j\in \mathcal M_j\\ t_1\equiv a\overline{t_2t_3}\pmod q \\ t_1\equiv v\pmod D}}1.$$
Let  $J=(q,D)$, $q'=q/J$, and $D'=D/J$.  Note that by the Chinese Remainder Theorem, if $m\equiv a\pmod q$ and $m\equiv b\pmod D$ and $a\equiv b\pmod J$ then $m\equiv aD'+bq'\pmod{q'D}$.  Applying (\ref{congclass01}) then gives
\begin{align*}&\sum_{0<v<D}\chi(v)\mathop{\sum\sum}\limits_{t_2,t_3\in \mathcal M_2,\mathcal M_3}\sum_{\substack{M_1\leq t_1\leq M_1+M_1' \\ t_1\equiv a\bar t_2\bar t_3 \pmod{q}\\ t_1\equiv v\pmod D}}\chi(t_1)\\
&=\sum_{0<v<D}\chi(v)\mathop{\sum\sum}\limits_{\substack{t_2,t_3\in \mathcal M_2,\mathcal M_3 \\ t_2t_3\equiv v\pmod{J}}}\left[\frac{M_1'}{q'D}+\sum_{0<|h|\leq H}c_{q'D}(h)e\left(-\frac{(a\bar t_2\bar t_3D' +vq')h}{Dq'}\right)\right]\\
&\phantom{=}+\sum_{0<v<D}\chi(v)\mathop{\sum\sum}\limits_{\substack{t_2,t_3\in \mathcal M_2,\mathcal M_3 \\ t_2t_3\equiv v\pmod J}}\left[O\left(\varrho\left(\frac{M_1+M_1'-(aD'\bar t_2\bar t_3+D'v)}{Dq'}\right)+\varrho\left(\frac{M_1-(aD'\bar t_2\bar t_3+D'v)}{Dq'}\right)\right)\right]
\end{align*}
For the first sum, we write $v=v_1J+v_2$, where $0\leq v_2<J$, and we similarly write $t_2=r_1J+r_2$.  So we have $v_2\equiv r_2t_3\pmod J$, and hence
\begin{align*}\sum_{0<v<D}&\chi(v)\mathop{\sum\sum}\limits_{\substack{t_2,t_3\in \mathcal M_2,\mathcal M_3 \\ t_2t_3\equiv v\pmod{J}}}1\\
=&\mathop{\sum}\limits_{t_3\in \mathcal M_3}\sum_{r_1=\lceil M_2/J\rceil}^{\lfloor (M_2+M_2')/J\rfloor -1}\sum_{r_2=0}^{J-1}\sum_{v_1=0}^{D'-1}\chi(v_1J+r_2t_3)+O\left(D^2M_3\right)\\
=&O\left(D^2M_3\right),
\end{align*}
since the character sum is now zero.  The bound for the sum of $t_2$ and $t_3$ over the $\varrho$ terms is the same as before, since $D'v$ is much smaller than $M_1$.  Hence,
\begin{align*}&\mathop{\sum\sum}\limits_{t_2,t_3\in \mathcal M_2,\mathcal M_3}\sum_{\substack{M_1\leq t_1\leq M_1+M_1' \\ t_1\equiv a\bar t_2\bar t_3 \pmod{q}\\ t_1\equiv v\pmod D}}\chi(t_1)\\
&=O\left(\frac{M_1M_3D^2}{q}\right)+\sum_{0<v<D}\chi(v)\mathop{\sum\sum}\limits_{\substack{t_2,t_3\in \mathcal M_2,\mathcal M_3 \\ t_2t_3\equiv v\pmod J}}\sum_{0<|h|\leq H}c_q(h)e\left(-\frac{(a\bar t_2\bar t_3D' +vq')h}{Dq'}\right)\\
&\phantom{=}+O\left(\left(\frac{M_2M_3}{H}+(Dq)^\frac 12+\frac{M_2M_3}{Dq'}\right)Dx^\varepsilon\right).
\end{align*}
Again, we can bound the big-O terms by $\frac{x^{1+\varepsilon}}{q^{1+\alpha}}+D^\frac 32q^{\frac 12}x^\varepsilon$.  So
\begin{align*}&\sum_{0<v<D}\chi(v)\mathop{\sum\sum}\limits_{t_2,t_3\in \mathcal M_2,\mathcal M_3}\sum_{\substack{M_1\leq t_1\leq M_1+M_1' \\ t_1\equiv a\bar t_2\bar t_3 \pmod{q}\\ t_1\equiv v\pmod D}}\chi(t_1)\\
&=\sum_{0<v<D}\chi(v)\mathop{\sum\sum}\limits_{\substack{t_2,t_3\in \mathcal M_2,\mathcal M_3 \\ t_2t_3\equiv v\pmod J}}\sum_{0<|h|\leq H}c_q(h)e\left(-\frac{(a\bar t_2\bar t_3D' +vq')h}{Dq'}\right)+O\left(\left(\frac{x}{q^{1+\alpha}}+D^2q^{\frac 12}\right)x^\varepsilon\right).
\end{align*}
We detect the remaining congruence condition by Dirichlet characters, finding
\begin{align*}&\sum_{0<v<D}\chi(v)\mathop{\sum\sum}\limits_{\substack{t_2,t_3\in \mathcal M_2,\mathcal M_3 \\ t_2t_3\equiv v\pmod J}}\sum_{0<|h|\leq H}c_q(h)e\left(-\frac{(a\bar t_2\bar t_3D' +vq')h}{Dq'}\right)\\
&=\sum_{0<v<D}\frac{\chi(v)}{\phi(J)}\sum_{\chi' \pmod J}\mathop{\sum\sum}\limits_{\substack{t_2,t_3\in \mathcal M_2,\mathcal M_3 }}\chi'(t_2)\chi'(t_3)\overline{\chi'(v)}\sum_{0<|h|\leq H}c_q(h)e\left(-\frac{a\bar t_2\bar t_3h}{q}\right)e\left(-\frac{vh}{D}\right)\\
&\phantom{=}+O\left(\frac{x}{q^{1+\alpha}}+D^2q^{\frac 12+\alpha}\right)
\end{align*}
To prove the last inequality of the lemma, note that
\begin{align*}
&\left|\sum_{0<v<D}\frac{\chi(v)}{\phi(J)}\sum_{\chi' \pmod J}\mathop{\sum\sum}\limits_{\substack{t_2,t_3\in \mathcal M_2,\mathcal M_3 }}\chi'(t_2)\chi'(t_3)\overline{\chi'(v)}\sum_{0<|h|\leq H}c_q(h)e\left(-\frac{a\bar t_2\bar t_3h}{q}\right)e\left(-\frac{vh}{D}\right)\right|\\
&\ll D\max_{\chi'\pmod J}\left|\mathop{\sum\sum}\limits_{\substack{t_2,t_3\in \mathcal M_2,\mathcal M_3 }}\chi'(t_2)\chi'(t_3)\sum_{0<|h|\leq H}c_q(h)e\left(-\frac{a\bar t_2\bar t_3h}{q}\right)e\left(-\frac{vh}{D}\right)\right|.
\end{align*}
So if $t_2$ or $t_3$ is twisted then we take $g_i(n)=\chi(n)$ for the $t_i$ that is twisted and $g_j(n)=1$ for the remaining $j\neq i$.  Meanwhile if $t_1$ is twisted then we take
\begin{gather*}
g_1(n)=e\left(-\frac{vn}{D}\right),\\
g_2(n)=g_3=\chi'(n).
\end{gather*}
Since all of these $g_i$ are 1-bounded, the lemma follows.
\end{proof}
\begin{corollary}
For $D^2q<x^{\frac 23-2\alpha}$, and for $\varepsilon<\alpha/4$, 
\begin{gather}\left|\mathcal D_{\mathcal M,f}\right|\ll D^2q^{-1+\alpha}M_1\max_{1\leq y\leq H}\left|\mathop{\sum\sum}\limits_{\substack{t_2,t_3\in \mathcal M_2,\mathcal M_3 \\ (t_2t_3,q)=1 }}\sum_{0<|h|\leq y}g_1(h)g_2(t_2)g_3(t_3)e\left(-\frac{a\overline{t_2} \overline{t_3} h}{q}\right)\right|+\frac{x}{q^{1+\frac \alpha 2}}
\end{gather}
\end{corollary}
\begin{proof}
We apply Abel's summation formula as before.
\end{proof}

\section{Applying the Shparlinski bounds}
Having proven that Theorems \ref{Shp:T1} and \ref{Shp:T2} are relevant to $\mathcal D_{\mathcal M,f}$, we now apply those theorems.
\begin{lemma}\label{ShpApp1} Let $0<\alpha<1/100$.  For $D^2q<x^{\frac 23-4\alpha}$ and $M_2\leq x^\frac 13$, and for any $\varepsilon$ with $0<\varepsilon<\alpha/4$,
\begin{align*}\left|\mathcal D_{\mathcal M,f}\right|\ll &D^3x^\varepsilon q^{3\alpha}M_2^\frac 98q^{-\frac 18}x^\frac 14+\frac{x}{q^{1+\frac{\alpha}{2}}}
\end{align*}
\end{lemma}
\begin{proof}
Applying Theorem \ref{Shp:T1}, we let $H$ be as above, $M=M_2$, and $N=M_3$.  Then
\begin{align*}\left|\mathcal D_{\mathcal M,f}\right|\ll &D^2x^\varepsilon q^{-1+\alpha}M_1\left(HM_2+(HM_2)^\frac 34q^\frac 14\right)\left(M_3^\frac 78q^{-\frac 18}+M_3^\frac 12\right)+\frac{x}{q^{1+\frac{\alpha}{2}}}\\
\ll &D^3q^{2\alpha}x^\varepsilon\left(M_2+\frac{M_2^\frac 34q^\frac 14}{H^\frac 14}\right)\left(M_3^\frac 78q^{-\frac 18}+M_3^\frac 12\right)+\frac{x}{q^{1+\frac{\alpha}{2}}}\\
\ll &D^3q^{2\alpha}x^\varepsilon\left(M_2+\frac{M_2^\frac 12x^\frac 14}{M_3^\frac 14}\right)\left(M_3^\frac 78q^{-\frac 18}+M_3^\frac 12\right)+\frac{x}{q^{1+\frac{\alpha}{2}}}.
\end{align*}
Recall that $M_3>\frac{x}{q^{\frac 32+\alpha}}$.  We can see that $$M_3^\frac 78q^{-\frac 18}\geq M_3^\frac 12$$ as long as $M_3\geq q^{\frac 16}$, which clearly holds here since $q<x^\frac{8}{15}$ and hence $\frac{x}{q^{\frac 32+\alpha}}>q^{\frac 16}$.  Moreover, since $M_3\leq M_2\leq x^\frac 13$ by assumption,
$$\frac{M_2^\frac 12x^\frac 14}{M_3^\frac 14}\geq M_2^\frac 14x^\frac 14\geq M_2.$$
So
\begin{align*}
\left|\mathcal D_{\mathcal M,f}\right|\ll &D^3x^\varepsilon q^{3\alpha}M_2^\frac 12M_3^\frac 58q^{-\frac 18}x^\frac 14+\frac{x}{q^{1+\frac{\alpha}{2}}}.
\end{align*}
Using the bound that $M_3\leq M_2$, we then have 
\begin{align*}
\left|\mathcal D_{\mathcal M,f}\right|\ll &D^3x^\varepsilon q^{3\alpha}M_2^\frac 98q^{-\frac 18}x^\frac 14+\frac{x}{q^{1+\frac{\alpha}{2}}}.
\end{align*}

\end{proof}




\begin{lemma}\label{ShpApp2}
Let $0<\alpha<1/100$, $D^2Q<x^{\frac 23-4\alpha}$, and $M_2\leq x^\frac 13$.  Let $\kappa>0$ be a fixed real number.  For all but at most $Q^{1-4\kappa+o(1)}$ values of $q\in [Q,2Q]$
\begin{align*}\max_{(a,q)=1}|\mathcal D_{\mathcal M,f}|\ll &D^3x^\varepsilon\left(M_2^\frac 54x^\frac 14Q^{-\frac 14}+M_2^\frac 34x^\frac 14\right)Q^{2\alpha+\kappa}+\frac{x}{Q^{1+\frac \alpha 2}}
\end{align*}
for any $\varepsilon$ with $0<\varepsilon<\alpha/4$.
\end{lemma}
\begin{proof}
We proceed as in the previous lemma, letting $M=M_2$ and $N=M_3$ and then factoring an $H$.  This gives
\begin{align*}\max_{(a,q)=1}|\mathcal D_{\mathcal M,f}|\ll &D^3x^\varepsilon\left(M_2+\frac{M_2^\frac 12x^\frac 14}{M_3^\frac 14}\right)\left(M_3Q^{-\frac 14}+M_3^\frac 12\right)Q^{2\alpha+\kappa+o(1)}+\frac{x}{Q^{1+\frac \alpha 2}}.
\end{align*}
We absorb the $Q^{o(1)}$ into the $x^\varepsilon$-term.  Again,
$$\frac{M_2^\frac 12x^\frac 14}{M_3^\frac 14}\geq M_2^\frac 14x^\frac 14\geq M_2.$$
So
\begin{align*}\max_{(a,q)=1}|\mathcal D_{\mathcal M,f}|\ll &D^3x^\varepsilon\left(M_2^\frac 12M_3^\frac 34x^\frac 14Q^{-\frac 14}+M_2^\frac 12M_3^\frac 14x^\frac 14\right)Q^{2\alpha+\kappa}+\frac{x}{Q^{1+\frac \alpha 2}}.
\end{align*}
Using the bound that $M_3\leq M_2$,
\begin{align*}\max_{(a,q)=1}|\mathcal D_{\mathcal M,f}|\ll &D^3x^\varepsilon\left(M_2^\frac 54x^\frac 14Q^{-\frac 14}+M_2^\frac 34x^\frac 14\right)Q^{2\alpha+\kappa}+\frac{x}{Q^{1+\frac \alpha 2}}.
\end{align*}

\end{proof}
\section{Proofs of Theorems \ref{MainSubThm1} and \ref{MainSubThm2}}

Finally, we prove Theorems \ref{MainSubThm1} and \ref{MainSubThm2}.  We begin with the former, which will follow from the following theorem.

\begin{theorem}\label{MainReprise1}  Let $\alpha$ be such that $0<\alpha<1/100$. If $D^{3}q<x^{\frac{30}{59}-4\alpha}$ then 
\begin{align*}
|\mathcal D_{\mathcal M,f}|\ll &\frac{x}{q^{1+\frac \alpha 2}}.
\end{align*}
\end{theorem}
\begin{proof}
We split the proof into two cases: $M_2\leq q^{\frac{3}{13}}x^\frac{2}{13}$ and $M_2>q^{\frac{3}{13}}x^\frac{2}{13}$.  

If $M_2>q^{\frac{3}{13}}x^\frac{2}{13}$ then we apply Lemma \ref{G:FI}, finding
\begin{gather}\label{q752}|\mathcal D_{\mathcal M,f}|\ll D^3x^\varepsilon q^{\frac 14+3\alpha} \left(\frac{x}{M_2}\right)^{\frac 12}+\frac{x}{q^{1+\frac \alpha 2}}\ll D^3x^\varepsilon q^{\frac{7}{52}+3\alpha}x^{\frac{22}{52}}+\frac{x}{q^{1+\frac \alpha 2}}.\end{gather}
We can assume that $\varepsilon<\frac{\alpha}{200}$.  Plugging in our bound for $q$ gives
\begin{align*}D^3q^{\frac{7}{52}+3\alpha}x^{\frac{22}{52}}
=&D^3x^\varepsilon q^{3\alpha} \frac xq\left(q^{59}x^{-30}\right)^\frac{1}{52}\\
\leq &D^3x^\varepsilon q^{3\alpha}\frac xq\left(\frac{x^{-118\alpha}}{D^{177}}\right)^\frac{1}{52}\\
\leq &D^3x^\varepsilon q^{3\alpha}\frac xq\left(\frac{q^{-118\left(\frac{52}{27}\right)\alpha}}{D^{177}}\right)^\frac{1}{52}\\
\ll &\frac{x}{q^{1+\alpha}}.
\end{align*}
Now let $M_2\leq q^{\frac{3}{13}}x^\frac{2}{13}$.  Since $q<x^\frac{2}{3}$, we have $M_2\leq x^{\frac 4{13}}<x^\frac 13$, and hence we can apply Lemma \ref{ShpApp1}.  So
\begin{align*}
|\mathcal D_{\mathcal M,f}|\ll &D^3x^\varepsilon q^{3\alpha}M_2^\frac 98q^{-\frac 18}x^\frac 14+\frac{x}{q^{1+\frac{\alpha}{2}}}\ll D^3x^\varepsilon q^{\frac{7}{52}+3\alpha}x^{\frac{22}{52}}+\frac{x}{q^{1+\frac{\alpha}{2}}}.
\end{align*}
This is the same bound as in (\ref{q752}), and hence the lemma follows.

\end{proof}

\begin{theorem}\label{MainReprise2}
Let $D^3Q<x^{\frac{16}{31}-2\alpha}$.  For all but at most $Q^{1-2\alpha+o(1)}$ values of $q\in [Q,2Q]$
\begin{align}\label{Qkappa}\max_{(a,q)=1}|\mathcal D_{\mathcal M,f}(a,q)|\ll &\frac{x}{Q^{1+\frac \alpha 2}}.
\end{align}
\end{theorem}

\begin{proof}
We split the proof into two cases.  Here, the split will be $M_2\leq Q^\frac 27x^\frac 17$ and $M_2>Q^\frac 27x^\frac 17$.  

First, assume $M_2\leq Q^\frac 27x^\frac 17$.  Applying Lemma \ref{G:FI}, we have
\begin{gather}\label{Q328}|\mathcal D_{\mathcal M,f}|\ll D^3x^\varepsilon q^{\frac 14+3\alpha} \left(\frac{x}{M_2}\right)^{\frac 12}+\frac{x}{q^{1+\frac{\alpha}{2}}}\ll D^3x^\varepsilon q^{\frac{3}{28}+3\alpha}x^{\frac{3}{7}}+\frac{x}{q^{1+\frac{\alpha}{2}}}.\end{gather}
We again assume that $\varepsilon<\frac{\alpha}{200}$.  Plugging in our bound for $q$ gives
\begin{align*}D^3x^\varepsilon q^{\frac{3}{28}+3\alpha}x^{\frac{6}{14}}\ll &D^3x^\varepsilon q^{3\alpha}\frac xq\left(q^{31}x^{-16}\right)^{\frac 1{28}}\\
\ll &D^3x^\varepsilon q^{3\alpha}\frac xq\left(D^{-93}x^{-62\alpha}\right)^{\frac 1{28}}\\
\ll &D^3x^\varepsilon q^{3\alpha}\frac xq\left(D^{-93}q^{-62\left(\frac{28}{15}\right)\alpha}\right)^{\frac 1{28}}\\
\ll &\frac{x}{q^{1+\alpha}}.
\end{align*}
Now, let $M_2<Q^\frac 27x^\frac 17$.  Since $Q<x^\frac{2}{3}$, we have $Q^\frac 27x^\frac 17<x^{\frac{4}{21}}x^\frac 17=x^\frac 13$.  Take $\kappa=\frac{\alpha}{2}$.  So we can apply Lemma \ref{ShpApp2}, finding that for all but at most $Q^{1-2\alpha+o(1)}$ values of $q\in [Q,2Q]$
\begin{align*}\max_{(a,q)=1}|\mathcal D_{\mathcal M,f}|\ll &D^3x^\varepsilon\left(Q^{\frac 3{28}}x^\frac{12}{28}+Q^{\frac{6}{28}}x^\frac{10}{28}\right)Q^{2\alpha+\kappa}+\frac{x}{q^{1+\frac{\alpha}{2}}}\ll D^3x^\varepsilon Q^{\frac 3{28}}x^\frac{12}{28}Q^{\frac 52\alpha}+\frac{x}{Q^{1+\frac \alpha 2}}.
\end{align*}
This is slightly smaller than the right-hand side of (\ref{Q328}), and so we can bound this in the same fashion.
\end{proof}
We prove the Elliott-Halberstam analogue as a corollary.
\begin{theorem}\label{MainReprise3}
Let $D^3Q<x^{\frac{16}{31}-2\alpha}$.  Then
\begin{align*}\sum_{q\sim Q}\max_{(a,q)=1}|\mathcal D_{\mathcal M,f}(a,q)|\ll &\frac{x}{Q^{\alpha}}.
\end{align*}
\end{theorem}
\begin{proof}
Let $\mathcal A$ denote the $q\in [Q,2Q]$ for which (\ref{Qkappa}) holds, and let $\mathcal B$ denote the remaining $q$.  Then
\begin{align*}\sum_{q\sim Q}\max_{(a,q)=1}|\mathcal D_{\mathcal M,f}(a,q)|=&\sum_{q\in \mathcal A}\max_{(a,q)=1}|\mathcal D_{\mathcal M,f}(a,q)|+\sum_{q\in \mathcal B}\max_{(a,q)=1}|\mathcal D_{\mathcal M,f}(a,q)|\\
&\ll \frac{x}{Q^{\frac{\alpha}{2}}}+\sum_{q\in \mathcal B}\max_{(a,q)=1}|\mathcal D_{\mathcal M,f}(a,q)|
\end{align*}
by the previous theorem, since $|\mathcal A|\leq Q$.

Note that trivially
$$\left|\mathcal D_{\mathcal M,f}(a,q)\right|\ll \sum_{\substack{n\leq 2x \\ n\equiv a\pmod q}}\tau_3(n)\ll \frac{x\log^3 x}{q}$$
by Shiu's theorem \cite{Sh}.  Since $|\mathcal B|\leq Q^{1-2\alpha+o(1)}$ by the theorem, we have
$$\sum_{q\in \mathcal B}\max_{(a,q)=1}|\mathcal D_{\mathcal M,f}(a,q)|\ll \frac{x}{Q^{2\alpha+o(1)}}\ll \frac{x}{Q^{\alpha}}.$$
Putting together our estimates for $\mathcal A$ and $\mathcal B$, we find
\begin{align*}\sum_{q\sim Q}\max_{(a,q)=1}|\mathcal D_{\mathcal M,f}(a,q)|\ll \frac{x}{Q^{\frac \alpha 2}},
\end{align*}
which is as required.
\end{proof}

\section{Acknowledgements}
I would like to thank Stylianos Sachpazis for pointing out an error in an early draft of the precursor to this paper.  I would also like to thank Igor Shparlinski for sending some helpful results about Kloosterman sums.  Thanks also to the anonymous referees for many helpful suggestions and corrections.

\bibliographystyle{line}

\end{document}